\documentclass[10pt,a4paper,reqno]{amsart}

\usepackage[utf8]{inputenc}
\usepackage[T1]{fontenc}
\usepackage{lmodern}

\usepackage[english]{babel}

\usepackage{amsmath}
\usepackage{amsfonts}
\usepackage{amssymb}
\usepackage{amsthm}

\usepackage{hyperref}

\theoremstyle{plain}
	\newtheorem{theorem}{Theorem}
	\newtheorem{proposition}[theorem]{Proposition}
	\newtheorem{lemma}[theorem]{Lemma}
	\newtheorem{corollary}[theorem]{Corollary}
\theoremstyle{definition}
	\newtheorem{definition}[theorem]{Definition}
	\newtheorem{example}[theorem]{Example}
\theoremstyle{remark}
	\newtheorem{remark}[theorem]{Remark}

\DeclareMathOperator{\Int}{Int}
\DeclareMathOperator{\End}{End}
\DeclareMathOperator{\supp}{supp}
\DeclareMathOperator{\Sym}{Sym}
\DeclareMathOperator{\diag}{diag}
\DeclareMathOperator{\Aut}{Aut}
\DeclareMathOperator{\Stab}{Stab}
\DeclareMathOperator{\Diag}{Diag}
\DeclareMathOperator{\id}{id}
\DeclareMathOperator{\rad}{rad}
\DeclareMathOperator{\Hom}{Hom}
\DeclareMathOperator{\Quad}{Quad}
\DeclareMathOperator{\Arf}{Arf}

\newcommand{\cR}{\mathcal{R}}
\newcommand{\cD}{\mathcal{D}}
\newcommand{\cV}{\mathcal{V}}
\newcommand{\ZZ}{\mathbb{Z}}
\newcommand{\FF}{\mathbb{F}}
\newcommand{\RR}{\mathbb{R}}
\newcommand{\CC}{\mathbb{C}}
\newcommand{\HH}{\mathbb{H}}
\newcommand{\idem}{\varepsilon}

\begin{document}

\title{Fine~gradings and automorphism~groups on associative~algebras}

\author{Adri\'an Rodrigo-Escudero}
\address{Departamento de Matem\'aticas y Computaci\'on,
Universidad de La Rioja, 26006 Lo\-gro\-\~no, Spain.}
\email{adrian.rodrigo.escudero@gmail.com}

\date{27 July 2020}

\begin{abstract}
First we prove that
any inner automorphism
in the stabilizer of
a graded-simple unital associative algebra
whose grading group is abelian
is the conjugation by a homogeneous element.
Now consider a grading by an abelian group
on an associative algebra
such that the algebra is graded-simple and
satisfies the DCC on graded left ideals.
We give necessary and sufficient conditions
for the grading to be fine.
Then we assume that
one of these necessary conditions to be fine
is satisfied,
and we compute the automorphism groups of the grading;
the results are expressed in terms of
the automorphism groups of a graded-division algebra.
Finally we compute the automorphism groups
of graded-division algebras
in the case in which
the ground field is the field of real numbers,
and the underlying algebra
(disregarding the grading)
is simple and of finite dimension.
\end{abstract}

\maketitle

\tableofcontents

\section{Introduction}

In recent years
the study of group gradings on different types of algebras
has become a very active research field.
Today the monograph \cite{EK2013}
is one of the main references on this topic.
There we find,
among many other things,
an overview of the progress done by numerous authors on this field,
and also a compilation
(and sometimes homogenization)
of the terminology about gradings that has appeared in the literature.
In this text we assume that
the basic definitions and properties of gradings
that appear in \cite{EK2013}
are known.

If we want to choose a type of algebra
to study its gradings,
associative algebras are a natural starting point.
Among the references that deal with this subject,
we can cite for example
\cite{BZ2002,Zol2002,BZ2003,BBK2012,BZ2016,Rod2016,BZ2018}.
Let us remark that this study is interesting
not only because
the good properties derived from the associative condition
make this question an approachable problem,
but also because the theory of affine group schemes
allows to transfer classifications of gradings
from associative algebras to nonassociative algebras.
Examples of this technique can be found in
\cite{BK2010,Eld2010,BKR2018b}.

One of the questions proposed
by the theory of gradings
is to determine if a grading is fine.
Besides being a natural question,
it is important to know the fine gradings of an algebra,
because the classification up to equivalence
of gradings that are not necessarily fine
is a difficult problem,
as it is illustrated in \cite[example 2.41 and figure 2.2]{EK2013}.
Therefore the classification of gradings up to equivalence
is usually restricted to fine gradings,
see for example
\cite{HPP1998,BSZ2001,CDM2010,Eld2010,DV2016,Eld2016,DE2017,EK2018}.
Once we have classified all gradings on a certain algebra,
a possible next step is to compute
the automorphism groups of those gradings.
For example
this has been done in
\cite{BK2010,EK2012}.

\medskip

The general objective of this article is
to characterize fine gradings on associative algebras,
and to compute their automorphism groups.
We will always assume that the algebras are graded-simple
and that the grading groups are abelian.

\medskip

More specifically
the main results and the structure of the article
are the following.
In section \ref{sect:inn_aut}
we prove theorem \ref{th:inn_aut},
which implies that
any inner automorphism
in the stabilizer of
a graded-simple unital associative algebra
whose grading group is abelian
is the conjugation by a homogeneous element.

Let $\cR$ be a graded-simple associative algebra
that satisfies the descending chain condition on graded left ideals
and such that its grading group is abelian.
By \cite[theorem 2.6 and equation (2.1)]{EK2013},
$\cR$ is a graded matrix algebra
over a graded-division algebra $\cD$.
The main result of section \ref{sect:fine_grad}
is theorem \ref{th:fine_grad},
which states that the grading on $\cR$ is fine
if and only if the division grading on $\cD$ is fine and
the fine condition of definition \ref{def:fine_cond} is satisfied.
As a consequence,
in remark \ref{rem:clas_equiv}
we observe that the classification,
up to equivalence,
of fine gradings by abelian groups
such that the algebra is graded-simple
and satisfies the descending chain condition on graded left ideals
is reduced to
the classification of fine division gradings.

In section \ref{sect:aut_groups} we assume that
the fine condition of definition \ref{def:fine_cond} is satisfied,
and we express the automorphism groups
of the graded algebra $\cR$
in terms of the automorphism groups
of the graded-division algebra $\cD$.
Propositions
\ref{pro:stab}, \ref{pro:diag_group} and \ref{pro:Weyl_group}
are respectively devoted to compute
the stabilizer, the diagonal group and the Weyl group of the grading.

Finally in sections \ref{sect:compl_grad} and \ref{sect:non-compl_grad}
we compute the automorphism groups
of the graded-division algebra $\cD$
in the case in which
the ground field is the field of real numbers,
and the underlying algebra $\cD$
(disregarding the grading)
is simple and of finite dimension.
We divide the analysis in two sections,
so that in section \ref{sect:compl_grad}
we study the case in which the graded real algebra $\cD$
can be regarded as a graded algebra
over the field of complex numbers,
and in section \ref{sect:non-compl_grad}
we study the rest of the cases.
The main results of these sections are
propositions \ref{pro:aut_2-f} and \ref{pro:stab_non_compl},
and remark \ref{rem:Weyl_non-compl}.

\section{Inner automorphisms}\label{sect:inn_aut}

\begin{theorem}\label{th:inn_aut}
Let $G$ be an abelian group,
$\FF$ a field,
and $\cR$ a $G$-graded unital associative $\FF$-algebra.
Assume that $\cR$ is graded-simple.
Let $ X \in \cR $ be an invertible element such that
the inner automorphism $ \Int X $ belongs to the stabilizer of the grading,
that is,
if $Y$ is a homogeneous element of $\cR$,
then $ ( \Int X ) (Y) = XYX^{-1} $ is also homogeneous of the same degree.
Then any nonzero homogeneous component $X_g$ of $X$ is invertible,
and determines the same inner automorphism as $X$,
that is,
$ \Int X_g = \Int X $.
\end{theorem}

\begin{proof}
For this first paragraph of the proof we follow Elduque's ideas from
\cite[lemma 10]{BKR2018a} or \cite[lemma 3.3]{Eld2010}.
Write $ \psi = \Int X $,
and note that $ \psi(Y)X = XY $ for all $ Y \in \cR $.
If $ Y \in \cR $ is homogeneous of degree $h$,
so it is $\psi(Y)$.
Since $G$ is abelian,
if we consider the component of degree $gh$
in the equation $ \psi(Y)X = XY $,
we get $ \psi(Y) X_g = X_g Y $ for all $ Y \in \cR_h $.
Therefore $ \psi(Y) X_g = X_g Y $ for all $ Y \in \cR $.
It remains to be shown that $X_g$ is invertible.

\medskip

Since $X$ is invertible,
there exists $ Z \in \cR $ such that $ X_g = XZ $.
For all $ Y \in \cR $
we have $ \psi(Y) X_g = X_g Y $,
so $ XYZ = XYX^{-1}XZ = XZY $,
hence $ YZ = ZY $.
Therefore $Z$ belongs to the center of $\cR$.

We have $ X^{-1} X_g = Z $,
and since $G$ is abelian,
we get that the homogeneous components of $ X^{-1} X_g $
also belong to the center.
So there exists $ W_0 \in \cR $ such that
$ W_0 X_g $ is a nonzero homogeneous central element.
In fact $ X^{-1} X_g = X_g X^{-1} $,
hence $ W_0 X_g = X_g W_0 $.

Now $\cR$ is graded-simple so,
by the graded version of Schur's lemma,
the nonzero homogeneous central element $ W_0 X_g $ is invertible.
Indeed,
$ ( W_0 X_g ) \cR $ is a graded two-sided ideal,
so $ ( W_0 X_g ) \cR = \cR $,
hence $ W_0 X_g $ is invertible.
Since $X_g$ and $W_0$ commute,
$X_g$ has both a left inverse and a right inverse,
and so $X_g$ is invertible.
\end{proof}

\begin{remark}
Recall that two invertible elements $X$ and $Y$
of a unital associative algebra $\cR$
define the same inner automorphism
if and only if
their quotient $ XY^{-1} $ belongs to
the center of the algebra $Z(\cR)$.
Recall also that,
if $\cR$ is endowed with a grading by an abelian group,
then the conjugation by any fixed homogeneous invertible element
defines an inner automorphism
that always belongs to the stabilizer of the grading.
Therefore if $\cR$ is graded-simple,
then theorem \ref{th:inn_aut} implies that
the group of inner automorphisms of $\cR$
that belong to the stabilizer of the grading
is isomorphic to
$ \cR_{\hom}^{\times} / ( Z(\cR) \cap \cR_{\hom}^{\times} ) $,
where $\cR_{\hom}^{\times}$ is the multiplicative group
of homogeneous invertible elements of $\cR$.
\end{remark}

\begin{example}
The graded-simple condition can not be removed
from the hypotheses of theorem \ref{th:inn_aut},
as the following counterexample shows.
Consider the decomposition of the real algebra
$ \HH \times \HH $
into the direct sum of seven subspaces
given by:
\begin{align}\label{eq:grad_HxH}
\HH \times \HH = {}
&
[ \RR (1,0) \oplus \RR (0,1) ] \oplus {}
\nonumber \\ &
\RR (i,0) \oplus
\RR (j,0) \oplus
\RR (k,0) \oplus {}
\nonumber \\ &
\RR (0,i) \oplus
\RR (0,j) \oplus
\RR (0,k)
\end{align}
Equation \eqref{eq:grad_HxH} defines a grading
by the abelian group $ \ZZ_2^2 \times \ZZ_2^2 $
in a natural way:
the two-sided ideal $ \HH \times \{ 0 \} $
is a graded subspace with support $ \ZZ_2^2 \times \{ e \} $,
and the two-sided ideal $ \{ 0 \} \times \HH $
is a graded subspace with support $ \{ e \} \times \ZZ_2^2 $.
In particular $ \HH \times \HH $ is not graded-simple.
If we take $ X = (i,i) $,
then the inner automorphism $ \Int X $
belongs to the stabilizer of the grading.
However the center of the algebra,
$ Z( \HH \times \HH ) =
\RR (1,0) \oplus \RR (0,1) $,
contains all the homogeneous elements of $ \HH \times \HH $
that are invertible.
\end{example}

\section{Fine gradings}\label{sect:fine_grad}

The notation that we introduce in this section
will also be used in section \ref{sect:aut_groups}.
Let $G$ be an abelian group,
$\FF$ a field,
and $\cR$ a $G$-graded associative $\FF$-algebra.
Assume that $\cR$ is graded-simple
and satisfies the descending chain condition on graded left ideals.
Then by \cite[theorem 2.6]{EK2013},
$\cR$ is isomorphic to $ \End_{\cD} (\cV) $,
where $\cD$ is a $G$-graded associative $\FF$-algebra
which is a graded-division algebra,
and $\cV$ is a $G$-graded right $\cD$-module
which is finite-dimensional over $\cD$.
Call $\Gamma$ the grading on $\cR$,
$\Gamma_0$ the division grading on $\cD$,
$ T \subseteq G $ the support of $\Gamma_0$,
and $\cD_{\hom}$ the multiplicative group
of nonzero homogeneous elements of $\cD$.

By \cite[section 2.1]{EK2013},
we can write $ \cV = \cV_1 \oplus \dots \oplus \cV_s $,
where the $\cV_i$ are the isotypic components of $\cV$.
Each of these components is determined by
its support $ g_i T \in G/T $
and its $\cD$-dimension $k_i$.
Write $ \gamma = (g_1,\dots,g_s) $,
$ \kappa = (k_1,\dots,k_s) $,
and $ k = k_1 + \dots + k_s $.
Recall that $ g_i \not\equiv g_j \pmod{T} $ for all $ i \neq j $.

\begin{definition}\label{def:fine_cond}
We say that $(\kappa,\gamma)$ satisfy the \emph{fine condition} if
$ k_1 = \dots = k_s = 1 $ and
$ g_i g_j^{-1} \not\equiv g_h g_{\ell}^{-1} \pmod{T} $
for all $ i \neq j $ and $ (i,j) \neq (h,\ell) $.
\end{definition}

Remark that, taking $h=\ell$,
the fine condition implies that
$ g_i g_j^{-1} \not\in T $ for all $ i \neq j $.
It also implies that $k=s$.

\medskip

Clearly if the grading $\Gamma$ on $\cR$ is fine,
then $(\kappa,\gamma)$ satisfy the fine condition and
the grading $\Gamma_0$ on $\cD$ is fine.
The main purpose of this section is to prove the converse
in theorem \ref{th:fine_grad}.
So we assume that $(\kappa,\gamma)$ satisfy the fine condition,
but we do not need to assume that the division grading $\Gamma_0$ is fine yet.

\medskip

We fix a homogeneous $\cD$-basis $(v_1,\dots,v_k)$ in $\cV$
such that $ \deg v_i = g_i $ for all $i$.
Then there is a bijective correspondence
between $ r \in \cR $ and $ X = (x_{ij}) \in M_k(\cD) $
given by $ r v_j = \sum_i v_i x_{ij} $.
It is also convenient to identify $M_k(\cD)$
with $ M_k(\FF) \otimes \cD $
via Kronecker product.
Thus the grading on $ M_k(\FF) \otimes \cD $ is given by:
\begin{equation}\label{eq:deg_Eij}
\deg ( E_{ij} \otimes d ) = g_i g_j^{-1} \deg d
\end{equation}

\begin{lemma}\label{lem:univ_group}
The universal abelian group of the grading $\Gamma$
is $ \ZZ^{k-1} \times T $.
\end{lemma}

\begin{proof}
Since $ g_i g_j^{-1} \not\equiv g_h g_{\ell}^{-1} \pmod{T} $
for all $ i \neq j $ and $ (i,j) \neq (h,\ell) $,
a homogeneous element of degree $ g_i g_j^{-1} t $
is necessarily of the form $ E_{ij} \otimes d $,
with $ \deg d = t \in T $.
Therefore the product of
a homogeneous element of degree $ g_i g_j^{-1} t_1 $
times a homogeneous element of degree $ g_h g_{\ell}^{-1} t_2 $
(where $ t_1,t_2 \in T $)
can be nonzero only if
either $i=j$ or $h=\ell$ or $j=h$.
So the grading $\Gamma$ admits a realization
by the group $ \ZZ^{k-1} \times T $.
Moreover,
the relations in $T$ are imposed by the grading $\Gamma$
because $\cD$ is a graded-division algebra
(consider for example $ E_{11} \otimes d_1 $ times $ E_{11} \otimes d_2 $),
so $ \ZZ^{k-1} \times T $ is
the universal abelian group of the grading $\Gamma$.
\end{proof}

Now we recall some properties about
homogeneous idempotents and graded left ideals.
Any homogeneous idempotent has degree $e$,
where $e$ is the neutral element of the group $G$.
By \cite[lemma 2.7]{EK2013}
any minimal graded left ideal of $\cR$
is generated by a homogeneous idempotent.

\begin{definition}
A \emph{primitive} homogeneous idempotent is
a nonzero homogeneous idempotent $ \varepsilon \in \cR $ such that
$ \cR \varepsilon $ is a minimal graded left ideal of $\cR$.
\end{definition}

Besides let us prove that
a nonzero homogeneous idempotent $ \varepsilon \in \cR $
is \emph{not} primitive
if and only if
there exist nonzero homogeneous idempotents $\delta$ and $\mu$
such that $ \varepsilon = \delta + \mu $
and $ \delta \mu = \mu \delta = 0 $.
The implication from right to left is clear,
moreover $ \cR \varepsilon = \cR \delta \oplus \cR \mu $.
For the implication from left to right take
a homogeneous idempotent $\rho$
such that $ 0 \neq \cR \rho \subsetneqq \cR \varepsilon $,
and check that we can take $ \delta = \varepsilon \rho $.

\begin{remark}\label{rem:prim_idemp}
There are $2^k$ homogeneous idempotents on
$ \cR \cong M_k(\FF) \otimes \cD $,
and exactly $k$ of them are primitive
($ E_{11} \otimes 1 , \dots , E_{kk} \otimes 1 $).
In fact $\cR$ is the direct sum of the $k$ minimal graded left ideals
generated by these primitive idempotents.
\end{remark}

\begin{theorem}\label{th:fine_grad}
The grading $\Gamma$ on $\cR$ is fine
in the class of abelian group gradings
if and only if
$(\kappa,\gamma)$ satisfy the fine condition and
the grading $\Gamma_0$ on $\cD$ is fine
in the class of abelian group gradings.
\end{theorem}

\begin{proof}
We have to prove the implication from right to left.
We assume that $(\kappa,\gamma)$ satisfy the fine condition,
but we do not need to assume that the division grading $\Gamma_0$ is fine yet.
Suppose that $\Gamma'$ is a refinement of the grading $\Gamma$,
and refining $\Gamma'$ if necessary,
we may also assume that the $(\kappa',\gamma')$ corresponding to $\Gamma'$
satisfy the fine condition,
and so lemma \ref{lem:univ_group} and remark \ref{rem:prim_idemp}
also apply to $\Gamma'$.

On the one hand,
if the universal abelian group of $\Gamma'$ is $U'$,
then the torsion-free rank of $U'$ is at least $k-1$,
so $ k' \geq k $.
On the other hand,
if $I$ is a minimal graded left ideal for $\Gamma'$,
so $I$ is generated by a homogeneous idempotent,
then $I$ is also graded for $\Gamma$,
and it is the direct sum of minimal graded left ideals for $\Gamma$.
Therefore $\cR$ is the direct sum of
$k'$ minimal graded left ideals for $\Gamma'$,
with $ k' \leq k $.
We conclude that $k=k'$,
and that the primitive homogeneous idempotents of $\Gamma$ and $\Gamma'$ coincide.

\medskip

By \cite[theorem 2.6 and lemma 2.8]{EK2013},
if $\idem$ is a primitive homogeneous idempotent of $\cR$,
both the graded-division algebra $\cD$
and its analogous $\cD'$ for $\Gamma'$
are isomorphic to $ \idem \cR \idem $.
Hereafter we assume that
the division grading $\Gamma_0$ on $\cD$ is fine,
hence the gradings $\Gamma$ and $\Gamma'$
coincide in $ \idem \cR \idem $.

\medskip

In the case of finite support,
both the number of homogeneous components of $\Gamma$ and $\Gamma'$
are $(k^2-k+1)|T|$,
so the refinement $\Gamma'$ is not a proper refinement of $\Gamma$.

\medskip

Now we finish the proof in the case of infinite support.
Since $\cR$ is isomorphic to $M_k(\cD')$,
we know that there exist homogeneous elements
$ \{ \idem_{ij} \mid 1 \leq i,j \leq k \} $
in $\cR$ such that
$ \idem_{ij} \idem_{ h \ell } = \delta_{jh} \idem_{ i \ell } $,
where $\delta_{jh}$ is $1$ if $j=h$ and $0$ otherwise.
Since the elements $\idem_{ij}$ are homogeneous in $\Gamma'$,
they are also homogeneous in $\Gamma$.
Necessarily $ \idem_{11} , \dots , \idem_{kk} $
are the primitive homogeneous idempotents,
and $ 1 = \idem_{11} + \dots + \idem_{kk} $.
We have a decomposition of $\cR$
as a direct sum of graded subspaces:
\begin{equation}
\cR = \bigoplus_{i,j} \idem_{ii} \cR \idem_{jj}
\end{equation}

Suppose that there exist
a nonzero homogeneous element in
$ \idem_{ii} \cR \idem_{jj} $ and
a nonzero homogeneous element in
$ \idem_{hh} \cR \idem_{ \ell \ell } $
whose degrees are different in $\Gamma'$ but equal in $\Gamma$,
and we will get a contradiction.
The fine condition implies that either $i=j$ and $ h = \ell $,
or $ i \neq j $ and $ (i,j) = (h,\ell) $.
In the first case,
we can assume that $i=h$ because
$ \idem_{ii} \cR \idem_{ii} $ and $ \idem_{hh} \cR \idem_{hh} $
are isomorphic via the isomorphism
$ X \in \idem_{ii} \cR \idem_{ii} \mapsto
\idem_{hi} X \idem_{ih} \in \idem_{hh} \cR \idem_{hh} $.
But this is a contradiction because we know that
the gradings $\Gamma$ and $\Gamma'$
coincide in $ \idem_{ii} \cR \idem_{ii} $.
In the second case,
both elements belong to the graded subspace
$ \idem_{ii} \cR \idem_{jj} $,
but $\Gamma$ and $\Gamma'$ also coincide in this subspace,
because it is equal to
$ \idem_{ij} \idem_{jj} \cR \idem_{jj} $,
another contradiction.
\end{proof}

\begin{remark}\label{rem:squares}
Since $(\kappa,\gamma)$ satisfy the fine condition,
given $ g \in \supp \Gamma $ there are two options:
\begin{itemize}
\item
If $ g \in T $,
then $ r^2 \neq 0 $ for all $ 0 \neq r \in \cR_g $.
\item
If $ g \notin T $,
then $ r^2 = 0 $ for all $ r \in \cR_g $.
\end{itemize}
\end{remark}

Let us speak about the development of theorem \ref{th:fine_grad}.
The first version of this result is due to Mikhail Kochetov,
who proved it in the case in which
the ground field is the field of real numbers and
the underlying algebra (disregarding the grading)
is simple and of finite dimension;
the techniques of this proof are totally different
from those that we have used.
Later on we proved the result in the general case,
with the hypotheses that we have stated at the beginning of the section.
Finally Alberto Elduque simplified the proof:
originally we used remark \ref{rem:squares}
instead of lemma \ref{lem:univ_group}
in order to prove that
the primitive homogeneous idempotents of $\Gamma$ and $\Gamma'$ coincide.
Let us state our arguments in the next two paragraphs,
because its techniques may be interesting.

\begin{proof}[Alternative proof]
Assume that a primitive homogeneous idempotent $\varepsilon$ for $\Gamma'$
can be further decomposed as $ \varepsilon = \delta + \mu $,
where $\delta$ is a primitive homogeneous idempotent for $\Gamma$,
$\mu$ is a nonzero homogeneous idempotent for $\Gamma$,
and $ \delta \mu = \mu \delta = 0 $.
By \cite[theorem 2.6 and lemma 2.8]{EK2013},
we may identify $\cD'$ with $ \varepsilon \cR \varepsilon $
and $\cD$ with $ \delta \cR \delta $,
so that $\cD$ is contained in $\cD'$
(in fact $ \varepsilon \cR \varepsilon =
\delta \cR \delta \oplus
\delta \cR \mu \oplus
\mu \cR \delta \oplus
\mu \cR \mu $).
In the next paragraph we are going to prove that $ \delta \cR \mu = 0 $,
therefore $\mu$ does not belong to the (graded) ideal generated by $\delta$,
which is a contradiction because $\cR$ is graded-simple.

Suppose that $ \delta \cR \mu \neq 0 $,
so there exists a nonzero homogeneous element in $ \delta \cR \mu $
of degree $ g \in \supp \Gamma $.
Since $ ( \delta \cR \mu )^2 = 0 $,
remark \ref{rem:squares} implies that $ g \notin T $.
Observe that $ \delta \cR \mu $ is contained in $ \varepsilon \cR \varepsilon $.
Applying again remark \ref{rem:squares} we conclude that
the square of any homogeneous element in $ \varepsilon \cR \varepsilon $
of degree $ g \in \supp \Gamma $ is zero.
However we can find a nonzero element in $ \varepsilon \cR \varepsilon $
such that it is homogeneous in $\Gamma'$
(so it is also homogeneous in $\Gamma$)
and such that its degree in $\Gamma$ is $g$;
but the square of this element is different from zero,
because $\cD'$ is a graded-division algebra,
contradiction.
\end{proof}

\begin{example}
Let us show an example that helped us
to visualize the proof of theorem \ref{th:fine_grad}.
Consider the following eight matrices of $M_2(\RR)$:
\begin{equation}
I = \begin{pmatrix} 1 & 0 \\ 0 & 1 \end{pmatrix}
\quad
X_a = \begin{pmatrix} 0 & 1 \\ 1 & 0 \end{pmatrix}
\quad
X_b = \begin{pmatrix} -1 & 0 \\ 0 & 1 \end{pmatrix}
\quad
X_c = \begin{pmatrix} 0 & -1 \\ 1 & 0 \end{pmatrix}
\end{equation}
\begin{equation}
E_{11} = \frac{1}{2} \begin{pmatrix} 1 & 1 \\ 1 & 1 \end{pmatrix}
\quad
E_{12} = \frac{1}{2} \begin{pmatrix} -1 & 1 \\ -1 & 1 \end{pmatrix}
\end{equation}
\begin{equation}
E_{21} = \frac{1}{2} \begin{pmatrix} -1 & -1 \\ 1 & 1 \end{pmatrix}
\quad
E_{22} = \frac{1}{2} \begin{pmatrix} 1 & -1 \\ -1 & 1 \end{pmatrix}
\end{equation}
We define a grading on the real algebra $M_2(\RR)$
by the group $\ZZ_2$ by means of:
\begin{align}
M_2(\RR) &
= [ \RR I \oplus \RR X_a ] \oplus [ \RR X_b \oplus \RR X_c ]
\nonumber \\ &
= [ \RR E_{11} \oplus \RR E_{22} ] \oplus [ \RR E_{12} \oplus \RR E_{21} ]
\end{align}
There are two ways to refine this grading
(see \cite[example 2.40]{EK2013}).
On the one hand we can construct
a division grading by the group $\ZZ_2^2$
by means of:
\begin{equation}
M_2(\RR) =
\RR I \oplus \RR X_a \oplus \RR X_b \oplus \RR X_c
\end{equation}
On the other hand we can construct
a grading by the group $\ZZ$
by means of:
\begin{equation}
M_2(\RR) =
\RR E_{21} \oplus [ \RR E_{11} \oplus \RR E_{22} ] \oplus \RR E_{12}
\end{equation}
\end{example}

\begin{remark}\label{rem:clas_equiv}
Suppose that $ \cR \cong M_k(\cD) $ and
$ \cR' \cong M_{k'}(\cD') $ are graded algebras
as described at the beginning of the section,
with their corresponding parameters
$(\kappa,\gamma)$ and $(\kappa',\gamma')$
satisfying the fine condition.
Then by \cite[proposition 2.33]{EK2013}
the graded algebras $\cR$ and $\cR'$ are equivalent
if and only if $ k = k' $ and
the graded-division algebras $\cD$ and $\cD'$ are equivalent.
Therefore the classification,
up to equivalence,
of fine gradings by abelian groups
such that the algebra is graded-simple
and satisfies the descending chain condition on graded left ideals
is reduced to
the classification of fine division gradings.
\end{remark}

\begin{example}
Let us enumerate all fine abelian group gradings
on the real algebra $M_4(\CC)$,
up to equivalence.
Recall that the division gradings on this real algebra
are classified in \cite{BZ2016} or \cite{Rod2016},
here we follow the list of \cite[section 3]{Rod2016}.
\begin{enumerate}
\item
$M_4(\cD)$ with $ \cD \cong \CC $ of type (1-c).
The universal abelian group is
$ \ZZ^3 \times \ZZ_2 $.
\item
$M_2(\cD)$ with $ \cD \cong M_2(\CC) $ of type (1-c).
The universal abelian group is
$ \ZZ \times \ZZ_2^3 $.
\item
$M_2(\cD)$ with $ \cD \cong M_2(\CC) $ of type (1-d).
The universal abelian group is
$ \ZZ \times \ZZ_2 \times \ZZ_4 $.
\item
$ \cD \cong M_4(\CC) $ of type (1-c).
The universal abelian group is
$ \ZZ_2^5 $.
\item
$ \cD \cong M_4(\CC) $ of type (1-d).
The universal abelian group is
$ \ZZ_2^3 \times \ZZ_4 $.
\item
$ \cD \cong M_4(\CC) $ of type (2-f)
and universal abelian group
$ \ZZ_4^2 $.
This is a complex grading.
\end{enumerate}
\end{example}

The difficulty of studying
gradings that are not necessarily fine
up to equivalence
is illustrated in \cite[example 2.41 and figure 2.2]{EK2013},
where the abelian group gradings
on the complex algebra $M_3(\CC)$
are classified up to equivalence;
there are two fine gradings,
but nine nontrivial gradings that are not necessarily fine.

\section{Automorphism groups}\label{sect:aut_groups}

In this section we are going to generalize
pages 45--49 of the monograph \cite{EK2013}
to the case in which the ground field $\FF$
is not necessarily algebraically closed.
So we are going to compute the automorphism groups of $\Gamma$.
In this analysis
we will use techniques similar to those in the monograph.
We continue with the notation that
we have introduced in section \ref{sect:fine_grad}.
Remark that we assume that $(\kappa,\gamma)$ satisfy the fine condition,
but we do not need to assume that the division grading $\Gamma_0$ is fine.

\medskip

Let $\Sym(k)$ be the symmetric group of degree $k$.
Recall that
the permutation matrix corresponding to $ \pi \in \Sym(k) $ is
the $ k \times k $ matrix $P(\pi)$ whose entries are
$1$ in the $ ( i , \pi^{-1} (i) ) $-th positions
and $0$ elsewhere.
Thus if $A$ is a $ k \times k $ matrix,
then the $i$-th row of $A$ equals
the $\pi(i)$-th row of $ P(\pi) A $,
and the $i$-th column of $A$ equals
the $ \pi^{-1} (i) $-th column of $ A P(\pi) $.
Hence $ P ( \pi \circ \sigma ) = P(\pi) P(\sigma) $
and $ ( P(\pi) )^{-1} = P ( \pi^{-1} ) $
for all $ \pi , \sigma \in \Sym(k) $.

\medskip

By \cite[proposition 2.33]{EK2013},
if $ \psi : \cR \to \cR $ is an equivalence of graded algebras,
then there exist
an equivalence of graded algebras $ \psi_0 : \cD \to \cD $
and an equivalence of graded vector spaces $ \psi_1 : \cV \to \cV $
such that for all $ r \in \cR $, $ d \in \cD $ and $ v \in \cV $
the following two equations are satisfied:
\begin{equation}\label{eq:equiv_D}
\psi_1(vd) = \psi_1(v) \psi_0(d)
\end{equation}
\begin{equation}\label{eq:equiv_R}
\psi_1(rv) = \psi(r) \psi_1(v)
\end{equation}
Besides,
another pair $ ( \psi_0' , \psi_1' ) $ satisfies
equations \eqref{eq:equiv_D} and \eqref{eq:equiv_R}
if and only if there exists $ d \in \cD_{\hom} $ such that
for all $ x \in \cD $ and $ v \in \cV $
we have:
\begin{equation}
\psi_0'(x) = d^{-1} \psi_0(x) d
\end{equation}
\begin{equation}
\psi_1'(v) = \psi_1(v) d
\end{equation}

\medskip

In section \ref{sect:fine_grad}
we fixed a homogeneous $\cD$-basis $(v_1,\dots,v_k)$ in $\cV$
such that $ \deg v_i = g_i $ for all $i$,
and so we obtained a bijective correspondence
between $ r \in \cR $ and $ X = (x_{ij}) \in M_k(\cD) $
given by $ r v_j = \sum_i v_i x_{ij} $.
Now we observe that
$ \psi (r) \psi_1 (v_j) = \sum_i \psi_1 (v_i) \psi_0 (x_{ij}) $,
that is,
the matrix of $ \psi(r) \in \cR $
in the basis $ ( \psi_1 (v_1) , \dots , \psi_1 (v_k) ) $
is $ \psi_0 (X) = ( \psi_0 (x_{ij}) ) \in M_k(\cD) $.
(Remark that $\psi_0$ acts on $X$ entrywise).
The matrix of change from the basis
$ ( \psi_1 (v_1) , \dots , \psi_1 (v_k) ) $
to the basis $(v_1,\dots,v_k)$
is $ \Psi = (d_{ij}) \in M_k(\cD) $,
where $ \psi_1(v_j) = \sum_i v_i d_{ij} $.
Therefore the bijective correspondence
sends $ \psi(r) \in \cR $
to $ \Psi \psi_0(X) \Psi^{-1} \in M_k(\cD) $.

Remark that $\psi_1$ is a graded map
and $ g_1 , \dots , g_k $ are different modulo $T$,
so in the sum $ \psi_1 (v_j) = \sum_i v_i d_{ij} $
only one of the coefficients $d_{ij}$ is different from zero.
Define $ \pi \in \Sym(k) $ so that
this nonzero coefficient is $ d_i := d_{ i \pi^{-1} (i) } $,
and note that $ d_i \in \cD_{\hom} $.
Then $ \Psi = DP $,
with $ D = \diag(d_1,\dots,d_k) $
and $ P = P(\pi) $.
Thus the bijective correspondence
sends $ \psi(r) \in \cR $
to $ D P \psi_0(X) P^{-1} D^{-1} \in M_k(\cD) $.

\medskip

Conversely,
suppose that we are given
a $ \psi_0 \in \Aut(\Gamma_0) $,
a $ k \times k $ permutation matrix $ P = P(\pi) $,
and a $ D = \diag(d_1,\dots,d_k) $
with $ d_i \in \cD_{\hom} $ for all $i$.
Then we can define an automorphism of $\FF$-algebras
$ M_k(\cD) \to M_k(\cD) $ by means of:
\begin{equation}
X \mapsto D P \psi_0(X) P^{-1} D^{-1}
\end{equation}
Recall that the degree of the generic homogeneous element
$ E_{ij} \otimes d \in M_k(\FF) \otimes \cD \cong M_k(\cD) $
is given by equation \eqref{eq:deg_Eij}.
Our map sends this element to the element
$ E_{ \pi(i) \pi(j) } \otimes d_{\pi(i)} \psi_0(d) d_{\pi(j)}^{-1} $,
which is homogeneous of degree:
\begin{align}\label{eq:deg_im_Eij}
\deg ( & E_{ \pi(i) \pi(j) } \otimes
d_{\pi(i)} \psi_0(d) d_{\pi(j)}^{-1} )
= \nonumber \\
& =
g_{\pi(i)} g_{\pi(j)}^{-1}
\deg d_{\pi(i)} ( \deg d_{\pi(j)} )^{-1} \deg \psi_0(d)
\end{align}
So this map is a graded map and,
since $(\kappa,\gamma)$ satisfy the fine condition,
it is also an equivalence of graded algebras.
Recall that the correspondence
between $ \psi \in \Aut(\Gamma) $
and $ ( (d_1,\dots,d_k) , \pi , \psi_0 )
\in (\cD_{\hom})^k \times \Sym(k) \times \Aut(\Gamma_0) $
is not one-to-one,
but we can only change $\psi_0$ by $ \Int(d^{-1}) \circ \psi_0 $,
and simultaneously change all $d_i$ by $ d_i d $,
where $ d \in \cD_{\hom}$.

\medskip

We will have a total description of the automorphism group $\Aut(\Gamma)$
if we see how this correspondence translates to the product.
Suppose that $\psi$ corresponds to $ ( D , \pi , \psi_0 ) $,
and $\psi'$ corresponds to $ ( D' , \pi' , \psi_0' ) $.
We define the action of $\pi$
on $ D = \diag(d_1,\dots,d_k) $ as
$ \pi(D) = \diag( d_{\pi^{-1}(1)} , \dots , d_{\pi^{-1}(k)} ) $,
that is,
the $i$-th element of $D$ equals
the $\pi(i)$-th element of $\pi(D)$.
Write $ P = P(\pi) $ and $ P' = P(\pi') $.
Then $ \psi \circ \psi' $ sends $X$ to:
\begin{align}
D P & \psi_0 ( D' P' \psi_0'(X) P'^{-1} D'^{-1} )
P^{-1} D^{-1} = \nonumber \\
& =
D P \psi_0(D') P'
\psi_0(\psi_0'(X))
P'^{-1} \psi_0(D')^{-1} P^{-1} D^{-1}
= \nonumber \\
& =
D \psi_0(\pi(D')) P P'
\psi_0(\psi_0'(X))
P'^{-1} P^{-1} \psi_0(\pi(D'))^{-1}  D^{-1}
\end{align}
Therefore $ \psi \circ \psi' $ corresponds to:
\begin{equation}
( D , \pi , \psi_0 ) * ( D' , \pi' , \psi_0' ) =
( D \cdot \psi_0(\pi(D')) , \pi \circ \pi' , \psi_0 \circ \psi_0' )
\end{equation}

\begin{proposition}\label{pro:stab}
The stabilizer of $\Gamma$ is isomorphic to:
\begin{equation}
\Stab(\Gamma) \cong ( \cD_e^{\times} )^{k-1} \rtimes \Stab(\Gamma_0)
\end{equation}
In this semidirect product,
the action of $\Stab(\Gamma_0)$
on $ ( \cD_e^{\times} )^{k-1} $
is componentwise.
\end{proposition}

\begin{proof}
We fix $ d_1 = 1 $ so that
the correspondence between $\psi$ and $ ( D , \pi , \psi_0 ) $
is one-to-one.
Because of equations \eqref{eq:deg_Eij} and \eqref{eq:deg_im_Eij},
if $ \deg d_2 = \dots = \deg d_k = e $,
$ \pi = \id $ and $ \psi_0 \in \Stab(\Gamma_0) $,
then $ \psi \in \Stab(\Gamma) $.
Conversely suppose that $ \psi \in \Stab(\Gamma) $;
considering $ E_{ii} \otimes d $ we get $ \psi_0 \in \Stab(\Gamma_0) $;
since $(\kappa,\gamma)$ satisfy the fine condition,
we get $ \pi = \id $;
finally considering $ E_{i1} \otimes d $
we get $ \deg d_i = e $ for all $i$.
\end{proof}

\begin{proposition}\label{pro:diag_group}
The diagonal group of $\Gamma$ is isomorphic to:
\begin{equation}
\Diag(\Gamma) \cong ( \FF^{\times} )^{k-1} \times \Diag(\Gamma_0)
\end{equation}
\end{proposition}

\begin{proof}
We extend the arguments of the proof of proposition \ref{pro:stab}.
Recall that $ \psi_0 (1) = 1 $.
If $ \psi \in \Diag(\Gamma) $,
considering $ E_{i1} \otimes 1 $
we get $ d_2 , \dots , d_k \in \FF^{\times} $,
and considering $ E_{11} \otimes d $
we get $ \psi_0 \in \Diag(\Gamma_0) $.
The action of $\psi_0$ on $\FF^{\times}$ is trivial
because $ \psi_0 (1) = 1 $.
\end{proof}

\begin{corollary}\label{cor:Stab=Diag}
If $ \dim \cD_e = 1 $,
then $ \Stab(\Gamma) = \Diag(\Gamma) $.
\qed
\end{corollary}

\begin{proposition}\label{pro:Weyl_group}
The Weyl group of $\Gamma$ is isomorphic to:
\begin{equation}
W(\Gamma) \cong T^{k-1} \rtimes ( \Sym(k) \times W(\Gamma_0) )
\end{equation}
In this semidirect product,
we identify $T^{k-1}$ with $ T^k / T $,
where $T$ is imbedded into $T^k$ diagonally;
$W(\Gamma_0)$ acts on $T^k$ componentwise,
and $ \pi \in \Sym(k) $ acts on $ \mathbb{T} = (t_1,\dots,t_k) \in T^k $
as $ \pi(\mathbb{T}) = ( t_{\pi^{-1}(1)} , \dots , t_{\pi^{-1}(k)} ) $,
that is,
the $i$-th element of $\mathbb{T}$ equals
the $\pi(i)$-th element of $\pi(\mathbb{T})$.
\end{proposition}

\begin{proof}
We can define an epimorphism of groups
$ f : \Aut(\Gamma) \to \Sym(k) \times W(\Gamma_0) $
because changing $\psi_0$ by $ \Int(d^{-1}) \circ \psi_0 $
does not change its class in $W(\Gamma_0)$.
Clearly $\Stab(\Gamma)$ is contained in $ K = \ker f $,
so we can consider
$ \bar{f} : W(\Gamma) \to \Sym(k) \times W(\Gamma_0) $.
The map $ K \to T^k / T $ that sends
$ ( (d_1,\dots,d_k) , \id , \psi_0 ) $ to
$ ( \deg d_1 , \dots , \deg d_k ) T $
is a well defined epimorphism of groups
whose kernel is precisely $\Stab(\Gamma)$;
therefore the kernel of $\bar{f}$ is isomorphic to $T^{k-1}$.

We have to show that the map
$ \bar{f} : W(\Gamma) \to \Sym(k) \times W(\Gamma_0) $
can be split.
From the homomorphism
$ \Sym(k) \times \Aut(\Gamma_0) \to \Aut(\Gamma) $
given by
$ ( \pi , \psi_0 ) \mapsto ( I_k , \pi , \psi_0 ) $,
we get a homomorphism
$ \Sym(k) \times \Aut(\Gamma_0) \to W(\Gamma) $
such that
$ \{ \id \} \times \Stab(\Gamma_0) $
is contained in its kernel,
so we obtain a section of $\bar{f}$.
\end{proof}

Propositions \ref{pro:stab}, \ref{pro:diag_group} and \ref{pro:Weyl_group}
generalize \cite[theorem 2.44]{EK2013}
to the case in which the ground field
is not necessarily algebraically closed
(see also \cite{EK2012}).
In the algebraically closed case,
the formulas for the diagonal group and the Weyl group remain the same,
whereas corollary \ref{cor:Stab=Diag} tells us that
the formula for the stabilizer is simplified
from a semidirect product to a direct product.

We can find an explicit description of
the action of the Weyl group of the grading
on its universal abelian group
in \cite[corollary 2.45]{EK2013}.
Remark that this corollary is a consequence of
equation \eqref{eq:deg_im_Eij}.

\medskip

The conclusion of this section is that,
in order to know the automorphism groups of $\cR$,
we have to study the automorphism groups of
the graded-division algebra $\cD$.
Clearly the diagonal group of $\Gamma_0$ is isomorphic to
the group of characters of the support $T$:
\begin{equation}
\Diag(\Gamma_0) \cong \Hom( T , \FF^{\times} )
\end{equation}
In sections \ref{sect:compl_grad} and \ref{sect:non-compl_grad}
we are going to compute
the stabilizer and the Weyl group of $\Gamma_0$
in the case in which
the ground field is the field of real numbers,
$ \FF = \RR $,
and the underlying algebra $\cD$
(disregarding the grading)
is simple and of finite dimension.

\section{Automorphism groups of complex division gradings}\label{sect:compl_grad}

Let $G$ be an abelian group,
$\cD$ a real algebra isomorphic to $M_n(\CC)$,
and $\Gamma_0$ a division $G$-grading on $\cD$.
Assume that the homogeneous components have dimension $2$,
and that the identity component $\cD_e$ coincides with the center $Z(\cD)$;
according to \cite[section 3]{Rod2016},
we say that the division grading $\Gamma_0$ is of type (2-f).
In this section we are going to compute
the automorphism groups of $\Gamma_0$.

This division grading $\Gamma_0$ can be regarded as
a grading of the complex algebra $M_n(\CC)$,
and these gradings are classified in \cite[theorem 2.15]{EK2013}.
The support $ T \subseteq G $ of $\Gamma_0$ is a group
isomorphic to $ \ZZ_{\ell_1}^2 \times \dots \times \ZZ_{\ell_r}^2 $,
where $ \ell_1 \cdots \ell_r = n $.
The commutation relations of homogeneous elements
define a $\CC$-valued alternating bicharacter $\beta$ on $T$
such that $ \rad(\beta) =
\{ t \in T \mid \beta(u,t)=1 , \, \forall u \in T \}
= \{ e \} $.
By a $\CC$-valued alternating bicharacter we mean a map
$ \beta : T \times T \to \CC^{\times} $ that satisfies
$ \beta(uv,w) = \beta(u,w) \beta(v,w) $,
$ \beta(u,vw) = \beta(u,v) \beta(u,w) $,
and $ \beta(u,u) = 1 $ for all $ u,v,w \in T $.
The commutation relations are explicitly the following;
for all $ X_u \in \cD_u $ and $ X_v \in \cD_v $ we have:
\begin{equation}\label{eq:comm_rel}
X_u X_v = \beta(u,v) X_v X_u
\end{equation}

\medskip

The equivalence classes of these complex gradings remain the same
if we regard them as gradings over the field of real numbers.
So two of these gradings $\Gamma_0$ and $\Gamma_0'$ are equivalent
if and only if their corresponding supports are isomorphic groups,
$ T \cong T' $.

\begin{remark}
In \cite[remark 18]{Rod2016} and \cite[section 8]{BKR2018a} we say that
the isomorphism classes of these complex gradings remain the same
if we regard them as gradings over the field of real numbers,
but this is not correct.
Two of these $G$-gradings $\Gamma_0$ and $\Gamma_0'$ are isomorphic
as complex gradings
if and only if their corresponding supports and alternating bicharacters coincide,
$ T = T' $ and $ \beta = \beta' $.
However,
since any isomorphism preserves the centers,
$\Gamma_0$ and $\Gamma_0'$ are isomorphic
as real gradings
if and only if $ T = T' $ and
either $ \beta = \beta' $ or $ \overline{\beta} = \beta' $.
\end{remark}

Recall from \cite{Rod2016}
(see also \cite{BZ2016})
that the division grading $\Gamma_0$ is fine
in the class of abelian group gradings
if and only if
the support $T$ is \emph{not}
an elementary abelian $2$-group,
if and only if $ \beta \neq \overline{\beta} $.

\begin{proposition}\label{pro:aut_2-f}
Suppose that the division grading $\Gamma_0$ is of type (2-f).
Then the Weyl group of $\Gamma_0$ is:
\begin{align}
W(\Gamma_0) = \{ f \in \Aut(T) \mid \text{either } &
\beta( f(u),f(v) ) = \beta(u,v) , \, \forall u,v \in T , \nonumber \\
\text{or } &
\beta( f(u),f(v) ) = \overline{\beta(u,v)} , \, \forall u,v \in T \}
\end{align}
Whereas the stabilizer of $\Gamma_0$ is isomorphic to:
\begin{itemize}
\item
If $ \beta \neq \overline{\beta} $,
then all automorphisms of $\Stab(\Gamma_0)$ are inner and:
\begin{equation}\label{eq:stab_2-f_fine}
\Stab(\Gamma_0) \cong \Hom( T , \CC^{\times} ) \cong T
\end{equation}
\item
If $ \beta = \overline{\beta} $,
then:
\begin{equation}\label{eq:stab_2-f_not_fine}
\Stab(\Gamma_0) \cong T \times \ZZ_2
\end{equation}
\end{itemize}
\end{proposition}

\begin{proof}
The Weyl group is easily obtained
because we know the invariants that determine
the isomorphism class of the grading.
Observe that the automorphisms may be
either of the first kind or of the second kind.

If $ \beta \neq \overline{\beta} $,
then there exist $ u,v \in T $ such that
$ \beta(u,v) \neq \overline{\beta(u,v)} $.
Therefore equation \eqref{eq:comm_rel} implies that
all automorphisms of $\Stab(\Gamma_0)$ are of the first kind,
that is, inner.
We have many ways to prove the isomorphisms of
equation \eqref{eq:stab_2-f_fine}.
$ \Stab(\Gamma_0) \cong \Hom( T , \CC^{\times} ) $
because all automorphisms of $\Stab(\Gamma_0)$ are of the first kind.
$ \Stab(\Gamma_0) \cong T $
because of theorem \ref{th:inn_aut}.
$ \Hom( T , \CC^{\times} ) \cong T $
because $\beta$ is nondegenerate.

If $ \beta = \overline{\beta} $,
then by \cite{Rod2016}
we can find a graded subalgebra $\cD_{\RR}$ of $\cD$ such that:
$\cD_{\RR}$ is a real algebra isomorphic to $M_n(\RR)$,
$\cD_{\RR}$ is a graded-division algebra
with homogeneous components of dimension $1$,
and the graded algebra $\cD$ is isomorphic to
the graded tensor product $ \cD_{\RR} \otimes \cD_e $:
\begin{equation}
\cD \cong \cD_{\RR} \otimes \cD_e
\end{equation}
Let $ \mu : T \to \{ \pm 1 \} $ be the quadratic form
associated to the graded-division algebra $\cD_{\RR}$.
This means that for every $ t \in T $ there exists
a homogeneous element $ X_t \in \cD_{\RR} $ of degree $t$
such that $ X_t^2 = \mu(t) I $,
where $I$ is the unity of $\cD$.
Moreover,
there are exactly two elements in $\cD$ of degree $t$
whose square is $ \mu(t) I $:
$X_t$ and $-X_t$.
Hence,
given a $ \psi \in \Stab(\Gamma_0) $,
the restriction of $\psi$ to $\cD_{\RR}$ is well defined,
and therefore $\Stab(\Gamma_0)$ is just the direct product
of the stabilizers of $\cD_{\RR}$ and $\cD_e$:
\begin{equation}
\Stab(\cD_{\RR}) \cong \Hom( T , \{ \pm 1 \} ) \cong T
\end{equation}
\begin{equation}
\Stab(\cD_e) \cong \Aut(\CC) \cong \ZZ_2
\end{equation}
\end{proof}

\begin{remark}
The isomorphism of equation \eqref{eq:stab_2-f_not_fine}
depends on the choice of the real form $\cD_{\RR}$
if $ n \geq 2 $.
However there is a way to encode the group $\Stab(\Gamma_0)$
without making any choice.
Denote by $\Quad(T,\beta)$ the set of quadratic forms
whose polarization is $\beta$,
that is:
\begin{equation}
\Quad(T,\beta) =
\{
\eta : T \to \{ \pm 1 \} \mid
\eta(uv) = \beta(u,v) \eta(u) \eta(v)
, \, \forall u,v \in T \}
\end{equation}
On the one hand
any automorphism $ \psi \in \Stab(\Gamma_0) $ of the first kind
corresponds to a character
$ \chi \in \Hom( T , \CC^{\times} ) = \Hom( T , \{ \pm 1 \} ) $
by means of the following equation,
where $ X_t \in \cD_t $:
\begin{equation}
\psi(X_t) = \chi(t) X_t
\end{equation}
On the other hand,
because of \cite[sections 8 and 11]{BKR2018a},
there is a distinguished involution $\varphi_0$ of the graded algebra $\cD$,
and it is of the second kind.
Moreover,
we can write any $ \psi \in \Stab(\Gamma_0) $ of the second kind as
$ \psi = \varphi \circ \varphi_0 = \varphi_0 \circ \varphi $,
where $\varphi$ is an involution of the graded algebra $\cD$
and of the first kind.
The involution $\varphi$ is encoded in a natural way
by a quadratic form $ \eta \in \Quad(T,\beta) $
by means of the following equation
(see \cite[equation (11)]{BKR2018a}),
where $ X_t \in \cD_t $:
\begin{equation}
\varphi(X_t) = \eta(t) X_t
\end{equation}
Therefore,
if $\Gamma_0$ is of type (2-f),
$ \beta = \overline{\beta} $
and $ n \geq 2 $,
then:
\begin{equation}
\Stab(\Gamma_0) \cong
\Hom( T , \{ \pm 1 \} ) \cup \Quad(T,\beta)
\end{equation}
Observe that this union is disjoint because $ n \geq 2 $.
Also note that,
since $ \varphi \circ \varphi_0 = \varphi_0 \circ \varphi $,
the product of two elements in this union set
is just the usual product of maps that take values in $ \{ \pm 1 \} $.
\end{remark}

\section{Automorphism groups of non-complex division gradings}\label{sect:non-compl_grad}

Let $G$ be an abelian group,
$\cD$ a finite-dimensional simple real associative algebra,
and $\Gamma_0$ a division $G$-grading on $\cD$.
These gradings are classified in \cite{BZ2016} or \cite{Rod2016},
here we follow the notation of \cite{Rod2016}.
Note that the real algebra $\cD$ is isomorphic to
either $M_n(\RR)$, $M_{n/2}(\HH)$ or $M_n(\CC)$;
also note that all homogeneous components of $\Gamma_0$
have the same dimension,
which can be $1$, $2$ or $4$,
according to the identity component $\cD_e$ being isomorphic to
$\RR$, $\CC$ or $\HH$.
In section \ref{sect:compl_grad} we have studied
the automorphism groups of $\Gamma_0$
in the case $ \cD_e = Z(\cD) \cong \CC $,
now we focus on the rest of the cases.

Let $I$ be the unity of $\cD$.
We denote by $T$ the support of $\Gamma_0$,
we define $ T^{[2]} = \{ t^2 \mid t \in T \} $,
and we denote by $K$
the support of the centralizer of the identity component,
$ K = \supp ( C_{\cD} (\cD_e) ) $.
If the homogeneous components have dimension $1$ or $4$
then $K=T$,
whereas if the homogeneous components have dimension $2$
and $ \cD_e \neq Z(\cD) $
then $K$ is a subgroup of $T$ of index $2$.
Let $ \beta : K \times K \to \{ \pm 1 \} $
be the alternating bicharacter
given by the commutation relations
in the centralizer of the identity component.
Finally we define the multiplicative group $\cD_{\hom}$
of nonzero homogeneous elements of $\cD$.

\begin{remark}\label{rem:Weyl_non-compl}
In \cite{Rod2016} the invariants
that characterize the isomorphisms classes of these gradings
are determined.
Note that these invariants are explicitly compiled in \cite{BKR2018a}.
As a consequence we immediately obtain their Weyl groups.
For example if $\Gamma_0$ is of type (2-b),
then by \cite[section 7]{BKR2018a}
the real algebra $\cD$ is isomorphic to $M_{n/2}(\HH)$
with $ n = 2^m \geq 2 $,
and the grading $\Gamma_0$ is determined up to isomorphism by $(T,K,\nu)$,
where $T$ is a subgroup of $G$ isomorphic to $\ZZ_2^{2m-1}$,
$K$ is a subgroup of $T$ of index $2$,
and $ \nu : T \setminus K \to \{ \pm 1 \} $
is a nice map such that
$\beta_{\nu}$ has type I and $ \Arf(\nu) = -1 $.
Therefore the Weyl group of $\Gamma_0$ in the case (2-b) is:
\begin{align}
W(\Gamma_0) & = \Aut(T,K,\nu) = \nonumber \\
& = \{ f \in \Aut(T) \mid
f(K) = K \text{ and } \nu(f(u)) = \nu(u)
, \, \forall u \in T \setminus K \}
\end{align}
\end{remark}

\begin{proposition}\label{pro:stab_non_compl}
The stabilizer of $\Gamma_0$ is isomorphic to:
\begin{itemize}
\item
If $\Gamma_0$ is of type (1-a), (1-b) or (1-c):
\begin{equation}\label{eq:stab_1-a-b-c}
\Stab(\Gamma_0) \cong \Hom( T , \{ \pm 1 \} ) \cong T
\end{equation}
\item
If $\Gamma_0$ is of type (1-d),
then all automorphisms of $\Stab(\Gamma_0)$ are inner and:
\begin{equation}\label{eq:stab_1-d}
\Stab(\Gamma_0) \cong \Hom( T , \{ \pm 1 \} ) \cong T / T^{[2]}
\end{equation}
\item
If $\Gamma_0$ is of type (2-a), (2-b) or (2-c):
\begin{equation}\label{eq:stab_2-a-b-c}
\Stab(\Gamma_0) \cong
( \CC^{\times} / \RR^{\times} )
\rtimes T
\end{equation}
In this semidirect product,
the elements of $ T \setminus K $
act on $ \CC^{\times} / \RR^{\times} $
by conjugation,
whereas the elements of $K$
act trivially on $ \CC^{\times} / \RR^{\times} $.
\item
If $\Gamma_0$ is of type (2-d) or (2-e),
then all automorphisms of $\Stab(\Gamma_0)$ are inner and:
\begin{equation}\label{eq:stab_2-d-e}
\Stab(\Gamma_0) \cong
( \CC^{\times} / \RR^{\times} )
\rtimes ( T / T^{[2]} )
\end{equation}
In this semidirect product,
the elements of $ ( T \setminus K ) / T^{[2]} $
act on $ \CC^{\times} / \RR^{\times} $
by conjugation,
whereas the elements of $ K / T^{[2]} $
act trivially on $ \CC^{\times} / \RR^{\times} $.
\item
If $\Gamma_0$ is of type (3-a), (3-b) or (3-c):
\begin{equation}\label{eq:stab_3-a-b-c}
\Stab(\Gamma_0) \cong
\Aut(\HH) \times \Hom( T , \{ \pm 1 \} ) \cong
\Aut(\HH) \times T
\end{equation}
\item
If $\Gamma_0$ is of type (3-d),
then all automorphisms of $\Stab(\Gamma_0)$ are inner and:
\begin{equation}\label{eq:stab_3-d}
\Stab(\Gamma_0) \cong
\Aut(\HH) \times \Hom( T , \{ \pm 1 \} ) \cong
\Aut(\HH) \times ( T / T^{[2]} )
\end{equation}
\end{itemize}
\end{proposition}

\begin{proof}
If the homogeneous components have dimension $1$,
then clearly:
\begin{equation}
\Stab(\Gamma_0) = \Diag(\Gamma_0) \cong
\Hom( T , \RR^{\times} ) = \Hom( T , \{ \pm 1 \} )
\end{equation}
Remark that in the cases (1-a) and (1-b)
the alternating bicharacter $\beta$ is nondegenerate,
and so it defines an isomorphism
between $ \Hom( T , \{ \pm 1 \} ) $ and $T$.
However in the case (1-c)
we know that $ \Hom( T , \{ \pm 1 \} ) \cong T $
because $T$ is an elementary abelian $2$-group,
but we have to make a choice in order to define this isomorphism.
Finally in the case (1-d) we do not have to make any choice;
any character $ \chi \in \Hom( T , \{ \pm 1 \} ) $
acts trivially on $T^{[2]}$,
that is,
all automorphisms of $\Stab(\Gamma_0)$ are inner,
so $ \Stab(\Gamma_0) \cong T / T^{[2]} $
because of theorem \ref{th:inn_aut}
and because $ \rad(\beta) = T^{[2]} $.

\medskip

If the homogeneous components have dimension $4$,
which corresponds to cases (3-a), (3-b), (3-c) and (3-d),
then $ \cD_e \cong \HH $ and
we can apply the double centralizer theorem
(see \cite[theorem 19]{Rod2016})
to express $\cD$ as the graded tensor product:
\begin{equation}
\cD \cong \cD_e \otimes C_{\cD}( \cD_e )
\end{equation}
So we obtain equations \eqref{eq:stab_3-a-b-c} and \eqref{eq:stab_3-d}
from equations \eqref{eq:stab_1-a-b-c} and \eqref{eq:stab_1-d}.

\medskip

If $\Gamma_0$ is of type (2-a) or (2-b),
then $ \Stab(\Gamma_0) \cong \cD_{\hom} / \RR^{\times} $
by theorem \ref{th:inn_aut}.
The degree defines a group epimorphism:
\begin{equation}
\overline{\deg} :
\cD_{\hom} / \RR^{\times} \to T
\end{equation}
The kernel of this epimorphism is:
\begin{equation}
\ker \overline{\deg} =
\cD_e^{\times} / \RR^{\times}
\cong \CC^{\times} / \RR^{\times}
\end{equation}
In order to see that this epimorphism can be split,
we apply \cite[proposition 20]{Rod2016}
to obtain a refinement of the division grading $\Gamma_0$
with homogeneous components of dimension $1$
and support $ \ZZ_2 \times T $.
Thus we obtain a subgroup of $ \cD_{\hom} / \RR^{\times} $
isomorphic to $ \ZZ_2 \times T $,
and keeping half of the elements
we obtain a section $ T \to \cD_{\hom} / \RR^{\times} $
of our epimorphism.
Note that this refinement construction,
and therefore also the isomorphism of
equation \eqref{eq:stab_2-a-b-c},
depend on the choice of
a nonzero homogeneous element $ X_g \in \cD_g $
of degree $ g \in T \setminus K $
(see \cite[remark 21]{Rod2016}).
Finally we compute the action of $T$
on $ \cD_e^{\times} / \RR^{\times} $;
if $ \cD_e = \RR I \oplus \RR J \cong \CC $
with $ J^2 = -I $,
then \cite[remark 21]{Rod2016} implies that
for any $ X_t \in \cD_t $ and $ a,b \in \RR $
we have:
\begin{equation}
X_t (aI+bJ) X_t^{-1} =
\begin{cases}
	aI+bJ & \text{if } t \in K \\
	aI-bJ & \text{if } t \in T \setminus K
\end{cases}
\end{equation}

\medskip

If $\Gamma_0$ is of type (2-c),
fix $ g \in T \setminus K $ and write,
as in \cite[proof of theorems 22 and 23]{Rod2016}:
\begin{equation}
\cD \cong ( \cD_e \oplus \cD_g )
\otimes C_{\cD} ( \cD_e \oplus \cD_g )
\end{equation}
Note that $ \cD_e \oplus \cD_g $
is isomorphic either to $M_2(\RR)$ or $\HH$
and its support is $ \{ e,g \} $;
and note that $ C_{\cD} ( \cD_e \oplus \cD_g ) $
is isomorphic to $M_{n/2}(\CC)$
and its support is $K$.
So this reduces this case to cases (2-a), (2-b) and (1-c).
Thus:
\begin{equation}
\Stab(\Gamma_0) \cong
( ( \CC^{\times} / \RR^{\times} ) \rtimes \{ e,g \} ) \times K
\cong
( \CC^{\times} / \RR^{\times} ) \rtimes T
\end{equation}
This argument can also be applied to case (2-d).

\medskip

Finally assume that $\Gamma_0$ is of type (2-e).
First we are going to prove that
any automorphism $ \psi \in \Stab(\Gamma_0) $ is inner,
that is, that $\psi$ is of the first kind.
Take any $ t \in T $ of order $4$,
and any normalized element $ X_t \in \cD_t $
(that is, $ X_t^4 = -I $).
By \cite[equation (11)]{Rod2016}
$ X_t^2 \in Z(\cD) $ and
$ \psi(X_t)^2 = X_t^2 $.
Therefore
$ \psi(X_t^2) = \psi(X_t)^2 = X_t^2 $,
and $\psi$ is of the first kind.

Now we make arguments similar to those of cases (2-a) and (2-b).
Theorem \ref{th:inn_aut} implies that
$ \Stab(\Gamma_0) \cong
\cD_{\hom} / ( Z(\cD) \cap \cD_{\hom} ) $.
Recall that $ \supp Z(\cD) = T^{[2]} $,
so the degree defines a group epimorphism:
\begin{equation}
\overline{\deg} :
\cD_{\hom} / ( Z(\cD) \cap \cD_{\hom} ) \to T / T^{[2]}
\end{equation}
The kernel of this epimorphism is
isomorphic to $ \CC^{\times} / \RR^{\times} $.
Taking a refinement of $\Gamma_0$
we get a subgroup of $ \cD_{\hom} / \RR^{\times} $
isomorphic to $ \ZZ_2 \times T $.
Since the grading group is abelian,
the center $Z(\cD)$ is a graded subspace of this refinement,
and its support is $ \{ \bar{0} \} \times T^{[2]} $.
Therefore we can get a subgroup of
$ \cD_{\hom} / ( Z(\cD) \cap \cD_{\hom} ) $
isomorphic to $ T / T^{[2]} $,
and a section of our epimorphism.
\end{proof}

Recall that the real algebra $\HH$ is central simple,
so all its automorphisms are inner,
that is,
$ \Aut(\HH) \cong \HH^{\times} / \RR^{\times} $,
and also recall that any automorphism of $\HH$
acts on the imaginary part of $\HH$ as a rotation,
so $ \Aut(\HH) \cong SO(3) $.
On the other hand
$ \CC^{\times} / \RR^{\times} $ is isomorphic to the unit circle,
$ \CC^{\times} / \RR^{\times} \cong SO(2) $,
and to the group of automorphisms of $\HH$ that act trivially on $\CC$,
that is,
$ \CC^{\times} / \RR^{\times} \cong \Aut_{\CC} (\HH) $.

\medskip

In the following examples we compute,
on some real algebras of low dimension,
all fine abelian group gradings
(up to equivalence of gradings)
and its automorphism groups
(up to isomorphism of groups).

\begin{example}
There are,
up to equivalence,
two fine abelian group gradings
on the real algebra $M_2(\RR)$.
\begin{enumerate}
\item
$M_2(\cD)$ with $ \cD \cong \RR $ trivially graded.
The universal abelian group is
$\ZZ$.
The Weyl group is
$ \Sym(2) \cong \ZZ_2 $.
The stabilizer is
$\RR^{\times}$.
\item
$ \cD \cong M_2(\RR) $ of type (1-a).
The universal abelian group is
$\ZZ_2^2$.
The Weyl group is
$\ZZ_2$.
The stabilizer is
$\ZZ_2^2$.
\end{enumerate}
\end{example}

\begin{example}
There is,
up to equivalence,
one fine abelian group grading
on the real algebra $\HH$.
\begin{enumerate}
\item
$ \cD \cong \HH $ of type (1-b).
The universal abelian group is
$\ZZ_2^2$.
The Weyl group is
$\Sym(3)$.
The stabilizer is
$\ZZ_2^2$.
\end{enumerate}
\end{example}

\begin{example}
There are,
up to equivalence,
three fine abelian group gradings
on the real algebra $M_2(\CC)$.
\begin{enumerate}
\item
$M_2(\cD)$ with $ \cD \cong \CC $ of type (1-c).
The universal abelian group is
$ \ZZ \times \ZZ_2 $.
The Weyl group is
$ \ZZ_2 \times \Sym(2)
\cong \ZZ_2^2 $.
The stabilizer is
$ \RR^{\times} \times \ZZ_2 $.
\item
$ \cD \cong M_2(\CC) $ of type (1-c).
The universal abelian group is
$\ZZ_2^3$.
The Weyl group is
$\Sym(3)$.
The stabilizer is
$\ZZ_2^3$.
\item
$ \cD \cong M_2(\CC) $ of type (1-d).
The universal abelian group is
$ \ZZ_2 \times \ZZ_4 $.
The Weyl group is
$\ZZ_2^2$.
The stabilizer is
$\ZZ_2^2$.
\end{enumerate}
\end{example}

\begin{example}
There are,
up to equivalence,
two fine abelian group gradings
on the real algebra $M_3(\CC)$.
\begin{enumerate}
\item
$M_3(\cD)$ with $ \cD \cong \CC $ of type (1-c).
The universal abelian group is
$ \ZZ^2 \times \ZZ_2 $.
The Weyl group is
$ \ZZ_2^2 \rtimes \Sym(3)
\cong \Sym(4) $.
The stabilizer is
$ (\RR^{\times})^2 \times \ZZ_2 $.
\item
$ \cD \cong M_3(\CC) $ of type (2-f).
This is a complex grading.
The universal abelian group is
$\ZZ_3^2$.
The Weyl group is
$ \Aut( \ZZ_3^2 )
\cong \operatorname{GL} (2,3)
= \{ A \in M_2 (\ZZ_3) \mid \det A \neq 0 \} $.
The stabilizer is
$\ZZ_3^2$.
\end{enumerate}
\end{example}

\section*{Acknowledgments}

I want to thank Alberto Elduque
for his constant and disinterested help,
and Mikhail Kochetov
for numerous fruitful discussions about gradings and mathematics in general.
Especially
I want to remark the contribution of
both Mikhail and Alberto
to theorem \ref{th:fine_grad}.

Supported by grant
MTM2017-83506-C2-1-P
of the Agencia Espa\~nola de Investigaci\'on (AEI)
and Fondo Europeo de Desarrollo Regional (FEDER).

\end{document}